\documentclass[a4paper]{amsart}

\usepackage{amsmath, amsfonts, amssymb, amsthm, amscd}
\usepackage{pinlabel}
\usepackage{graphicx, color}

\theoremstyle{plain}
\newtheorem{theorem}{Theorem}[section]

\newtheorem{proposition}[theorem]{Proposition}
\newtheorem{corollary}[theorem]{Corollary}
\newtheorem{lemma}[theorem]{Lemma}

\newtheorem*{CT_theorem}{Cannon-Thurston Theorem}
\newtheorem*{cts_extension_theorem}{Continuous Extension Theorem}
\newtheorem*{closed_orbits_theorem}{Closed Orbits Theorem}

\theoremstyle{definition}

\theoremstyle{remark}

\newcommand{\bR}{\mathbb{R}}

\newcommand{\bH}{\mathbb{H}}

\newcommand{\cD}{\mathcal{D}}
\newcommand{\cE}{\mathcal{E}}

\newcommand{\orb}{P}
\newcommand{\Orb}{\overline{P}}

\newcommand{\rF}{\mathfrak{F}}

\newcommand{\Fr}{\text{Fr }}

\title{Quasigeodesic flows and sphere-filling curves}
\author{Steven Frankel}
\address{DPMMS \\ University of Cambridge \\ Cambridge CB3 0WA \\ United Kingdom}
\email{sf451@cam.ac.uk}
\date{\today}

\begin{document}

\begin{abstract}
Given a closed hyperbolic 3-manifold $M$ with a quasigeodesic flow we construct a $\pi_1$-equivariant sphere-filling curve in the boundary of hyperbolic space. Specifically, we show that any complete transversal $P$ to the lifted flow on $\bH^3$ has a natural compactification as a closed disc that inherits a $\pi_1$ action. The embedding $P \hookrightarrow \bH^3$ extends continuously to the compactification and the restriction $S^1 \to S^2_\infty$ to the boundary is a surjective $\pi_1$-equivariant map. This generalizes the result of Cannon and Thurston for fibered hyperbolic 3-manifolds.
\end{abstract}

\maketitle
\tableofcontents

\section{Introduction}
\subsection{Background and motivation}
Cannon and Thurston \cite{CannonThurston} showed that a hyperbolic 3-manifold fibering over the circle gives rise to a $\pi_1$-equivariant sphere-filling curve. The lift of a fiber to the universal cover is an embedded plane $P \hookrightarrow \bH^3$ which can be compactified to a disc $P \cup S^1_\infty$ in such a way that the embedding extends continuously to a map $P \cup S^1_\infty \to \bH^3 \cup S^2_\infty$, and the restriction $S^1_\infty \to S^2_\infty$ to the boundary is surjective and $\pi_1$-equivariant. 

The proof depends on the interaction between the geometry of the surface and the geometry of a certain transverse flow. A surface bundle over the circle is determined up to homeomorphism by its monodromy. If the resulting manifold is hyperbolic then the monodromy can be taken to be pseudo-Anosov, so its suspension gives rise to a pseudo-Anosov flow. It turns out that this flow is quasigeodesic in the sense that each lifted flowline is a quasi-isometrically embedded line in $\bH^3$, and one can use this fact to see that leaves of the stable/unstable foliations are quasi-isometrically embedded hyperbolic planes. Cannon and Thurston used the separation properties of these planes to prove the continuous extension property for the lifted surface fibers.

It turns out, and the main goal of this paper is to show, that the continuous extension property can be established in far greater generality and assuming far less initial structure. For any quasigeodesic flow on a hyperbolic 3-manifold, the lifted flow on the universal cover admits a complete planar transversal, i.e. a planar transversal that intersects every orbit. We show for {\em any} quasigeodesic flow that {\em any} complete transversal has the continuous extension property --- it can be compactified to a closed disk whose boundary circle maps continuously and surjectively to the sphere at infinity of hyperbolic 3-space. And although the transversal itself will typically have trivial stabilizer in $\pi_1$, its boundary admits a natural $\pi_1$-action, and the continuous surjection to $S^2_\infty$ is $\pi_1$-equivariant. Thus one obtains group invariant sphere-filling curves directly from the quasigeodesic property without appealing to a surface-bundle structure (which does not in general exist).

Our continuous extension theorem thus generalizes the Cannon-Thurston theorem to arbitrary quasigeodesic flows. Any finite depth taut foliation of a hyperbolic 3-manifold admits an almost-transverse quasigeodesic flow, so such flows can be found in abundance. Quasigeodesic flows are found even on non-fibered manifolds; in fact there are examples on homology spheres. Our argument is fundamentally different from the Cannon-Thurston argument, which makes use of 2-dimensional objects --- the surface fibers and the 2-dimensional leaves of stable/unstable foliations --- to reason about topology in 3-dimensions. By contrast, we work in the orbit space of the flow (a 2 manifold), and need to work with much more complicated objects, namely arbitrary closed continua, which are derived indirectly from the global geometry and dynamics of the transversal, and which hint at a kind of ``pseudo-Anosov structure in the large''. In fact, our main theorem may be thought of as a key step in Calegari's program to show that every quasigeodesic flow on a hyperbolic 3-manifold is homotopic (through quasigeodesic flows) to a quasigeodesic \emph{pseudo-Anosov} flow. Finally, our main theorem strengthens the main result of [3] by weakening the hypothesis under which a quasigeodesic flow can be shown to have a closed orbit.

\subsection{Statement of results}
A quasigeodesic is a quasi-isometric embedding of $\bR$ in a metric space $X$. That is, distances as measured in $X$ and $\bR$ are comparable on the large scale up to a multiplicative constant (see Section~\ref{sec:QuasigeodesicFlows} for a precise definition).

A nonsingular flow on a manifold is said to be quasigeodesic if the lift of each flowline to the universal cover is a quasigeodesic. Fix a quasigeodesic flow $\rF$ on a closed hyperbolic 3-manifold $M$. We can lift to a flow $\widetilde{\rF}$ on the universal cover $\widetilde{M} \simeq \bH^3$. The orbit space $P$ is the image of the quotient map $\pi:\bH^3 \to P$ that collapses each flowline in $\widetilde{\rF}$ to a point. This space is a plane and it comes with natural maps $e^\pm: P \to S^2_\infty$ that send each orbit to its forward/backward limit in the sphere at infinity. There is a natural compactification $\overline{P}$ of the orbit space as a closed disc called the \emph{end compactification} where the boundary is Calegari's universal circle $S^1_u$.

\begin{cts_extension_theorem}
Let $\rF$ be a quasigeodesic flow on a closed hyperbolic 3-manifold $M$. There are unique continuous extensions of the endpoint maps $e^\pm$ to the compactified orbit space $\overline{P}$, and $e^+$ agrees with $e^-$ on the boundary.
\end{cts_extension_theorem}

A complete transversal to the flow $\rF$ can be identified with the orbit space $P$, so we will think of it as a section $\psi: P \to S^2_\infty$ of the quotient map $\pi$. Thus it is immediate that any transversal can be compactified by adding the universal circle $S^1_u$. It is easy to show that we can extend $\psi$ by setting its values on the boundary equal to that of the $e^\pm$ maps and our generalization of the Cannon-Thurston theorem follows.

\begin{CT_theorem}
Let $\rF$ be a quasigeodesic flow on a closed hyperbolic 3-manifold $M$ and let $\psi:P \to \bH^3$ be a transversal to the lifted flow on the universal cover. There is a natural compactification of $P$ as a closed disc $P \cup S^1_u$ that inherits a $\pi_1$ action and a unique continuous extension $\psi:P \cup S^1_u \to \bH^3 \cup S^2_\infty$. The restriction of $\psi$ to $S^1_u$ is a $\pi_1$-equivariant space-filling curve in $S^2_\infty$.
\end{CT_theorem}

A group of homeomorphisms of the circle is said to be \emph{M\"obius-like} if each element is topologically conjugate to a M\"obius transformation, i.e. an element of $PSL(2, \bR)$ acting in the standard way on $\bR P^1$. It is said to be \emph{hyperbolic M\"obius-like} if each element is conjugate to a hyperbolic M\"obius transformation. We can use the extension theorem to improve on the main result in \cite{Frankel}.

\begin{closed_orbits_theorem}
Let $\rF$ be a quasigeodesic flow on a closed hyperbolic 3-manifold $M$. Suppose $\rF$ has no closed orbits. Then $\pi_1(M)$ acts on the universal circle $S^1_u$ as a hyperbolic M\"obius-like group. 
\end{closed_orbits_theorem}

Previously, we could only conclude that such a flow would have a ``rotationless M\"obius-like'' universal circle action, meaning that each element is conjugate to either a hyperbolic or parabolic M\"obius transformation.

\subsection{Future directions}
Given a closed hyperbolic 3-manifold $M$ with a quasigeodesic flow $\rF$, take two copies, $\overline{P}^+$ and $\overline{P}^-$, of the compactified orbit space and glue them together along the boundary. We'll call the resulting space the \emph{universal sphere} $S^2_u$. Since the $e^+$ and $e^-$ maps agree on the boundary one can define a map
\[ e: S^2_u \to S^2_\infty \]
that agrees with $e^+$ on $\overline{P}^+$ and $e^-$ on $\overline{P}^-$.

In addition, one can fill in the universal sphere to form a closed ball $B^3_u$ with a 1-dimensional foliation $\rF_u$ by connecting each point in $P^+$ to the corresponding point in $P^-$ with a line. This parallels Cannon and Thurston's construction for fibered hyperbolic manifolds in \cite{CannonThurston} and Fenley's ``model pre-compactification'' for hyperbolic manifolds with a pseudo-Anosov flow in \cite{Fenley}.

There is a natural decomposition of $S^2_u$ by point preimages $\{ e^{-1}(p) \}$. The fundamental group acts on $B^3_u$ preserving this decomposition and the foliation $\rF_u$. After collapsing each decomposition element to a point one recovers the \emph{topological} structure of the fundamental group acting on $\bH^3 \cup S^2_\infty$ and the foliation $\rF$.

A goal of Calegari's program is to show that one can deduce the pseudo-Anosov property of such a flow from the dynamics on the universal sphere. We would like to isotope a quasigeodesic flow in such a way that each decomposition element, restricted to $P^+$ or $P^-$, is a union of leaves of a singular foliation with ``pseudo-Anosov dynamics'' and conclude that the resulting flow is both quasigeodesic and pseudo-Anosov.

\subsection{Acknowledgements}
The arguments in this paper have gone through several iterations, each much longer and more technical than the present, and I am indebted to Danny Calegari for his careful and patient critique of each one. I would also like to thank Andr\'e de Carvalho, Sergio Fenley, and Dongping Zhuang for helpful discussions.

\section{Quasigeodesic flows and sphere-filling curves}
\subsection{Quasigeodesic flows} \label{sec:QuasigeodesicFlows}
Fix a closed hyperbolic 3-manifold $M$ with a quasigeodesic flow $\rF$.

We recall that this means that the lift $\gamma:\bR \to \bH^3$ of each flowline in the universal cover satisfies
\[ 1/k \cdot d(\gamma(x), \gamma(y)) - \epsilon \leq |x - y| \leq k \cdot d(\gamma(x), \gamma(y)) + \epsilon \]
for some constants $k, \epsilon > 0$ for all $x, y \in \bR$. We will not work directly with this metric property but with its topological consequences reviewed below. See \cite{Frankel}, Sections~2 \& 3 or \cite{Calegari} for a more detailed discussion.

We can lift $\rF$ to a flow $\widetilde{\rF}$ on the universal cover $\widetilde{M} \simeq \bH^3$. Each flowline in $\widetilde{\rF}$ has well-defined and distinct positive and negative endpoints in $S^2_\infty = \partial \bH^3$, and there is a constant $C$ such that each flowline has Hausdorff distance at most $C$ from the geodesic with the same endpoints.\footnote{\cite{BridsonHaefliger}, pp. 399-404} 

The orbit space $P$ of $\rF$ is the image of the map
\[ \pi: \bH^3 \to P \]
that collapses each flowline in $\widetilde{\rF}$ to a point. The orbit space is homeomorphic to the plane,\footnote{\cite{Calegari}, Theorem~3.12} and the action of $\pi_1(M)$ on $\bH^3$ by deck transformations descends to an action on $P$. The maps
\[ e^\pm: P \to S^2_\infty \]
that send each point to the positive/negative endpoint of the corresponding flowline are continuous  and have dense image.\footnote{\cite{Calegari}, Lemmas~4.3 \& 4.4} Since the endpoints of each flowline are distinct, $e^+(p) \neq e^-(p)$ for every $p \in P$.

\subsection{Sequences in $P$}
If we think about $\bH^3$ in the unit ball model we can endow $\overline{\bH^3} = \bH^3 \cup S^2_\infty$ with the Euclidean metric. This will be useful when considering sets that are ``close to the infinity.''

The following lemma says that sets that are close to infinity in $P$ have uniformly small diameter in $\overline{\bH^3}$.

\begin{lemma} \label{lemma:FlowlinesEventuallySmall}
For each $\epsilon > 0$ there is a compact set $A \subset \orb$ such that for each $p \in P \setminus A$, the diameter of $\pi^{-1}(p)$ is at most $\epsilon$ in the Euclidean metric.
\end{lemma}
\begin{proof}
If $B \in \bH^3$ is a large compact ball then each geodesic not intersecting $B$ has small diameter in the Euclidean metric. Since each flowline has Hausdorff distance at most $C$ from a geodesic with the same endpoints, the same applies to flowlines.

Given $\epsilon > 0$ choose a compact ball $B \in \bH^3$ such that each flowline not intersecting $B$ has diameter less than $\epsilon$ and set $A = \pi(B)$.
\end{proof}

In particular, points close to infinity in $P$ correspond to flowlines with endpoints uniformly close together, i.e for each $\epsilon > 0$ there is a compact set $A \subset P$ such that for each $p \in P \setminus A$, the distance between $e^+(p)$ and $e^-(p)$ is at most $\epsilon$.

Recall that if $(A_i)_{i=1}^\infty$ is a sequence of sets in a topological space then $\limsup A_i$ is the set of points $x$ such that every neighborhood of $x$ intersects infinitely many of the $A_i$.

If $Y$ is a subset of a topological space $X$ then the frontier of $Y$ is defined by $\Fr Y := \overline{Y} \cap \overline{X \setminus Y}$.

\begin{lemma} \label{lemma:LimSups}
Let $(U_i)_{i=1}^\infty$ be a sequence of disjoint open sets in $P$ with frontiers $A_i = \Fr U_i$. Then 
\[ \limsup e^+(U_i) \subseteq \limsup e^+(A_i). \]
\end{lemma}
\begin{proof}
It suffices to show that if $(x_i \in U_i)_{i=1}^\infty$ is a sequence of points such that $e^+(x_i)$ converges, where we might have already replaced the $U_i$ with a subsequence, then $\lim e^+(x_i) \in \limsup e^+(A_i)$.

Fix such a sequence and set $p := \lim e^+(x_i)$ and $Q := \limsup e^+(A_i)$. We need to show that $p \in Q$.

{\bf Case 1:} Suppose that infinitely many of the $x_i$ are contained in a bounded set. Then after taking a subsequence we can assume that the $x_i$ converge to some point $x$ and note that $e^+(x) = p$ (see Figure~\ref{figure:LimSupsCase1}). Then since $A_i$ separates $x_i$ from $x_{i-1}$ for each $i$ we can find points $y_i \in A_i$ that also converges to $x$. But then $\lim e^+(y_i) = p$, so $p \in Q$ as desired.

\begin{figure}[h]
	\labellist
	\small\hair 2pt
	\pinlabel $x_1$ [b] at 20 35
	\pinlabel $x_2$ [b] at 54 35
	\pinlabel $x_3$ [b] at 79 35
	\pinlabel $x$ [b] at 151 35
	
	\pinlabel $U_1$ [c] at 20 119
	\pinlabel $U_2$ [c] at 54 119
	\pinlabel $U_3$ [c] at 78 119
	
	\pinlabel $A_1$ [tr] at 16 11
	\pinlabel $A_2$ [tr] at 50 11
	\pinlabel $A_3$ [tr] at 76 11
	
	\pinlabel $\cdots$ [c] at 98 33
	\pinlabel $\cdots$ [c] at 98 119
	
	\pinlabel $\cdots$ [c] at 142 33
	\pinlabel $\cdots$ [c] at 142 119

	\endlabellist
	\centering
		\includegraphics{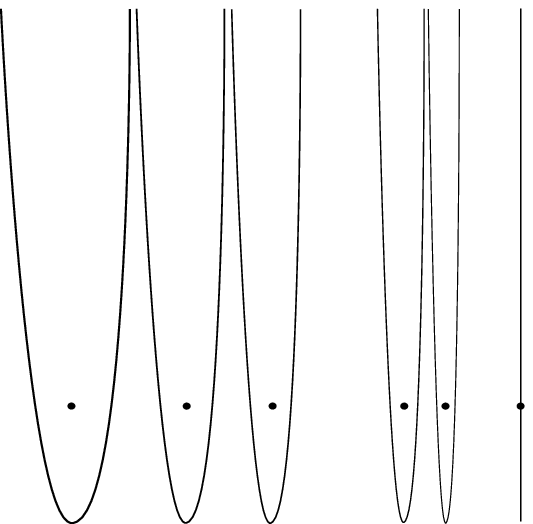}
	\caption{Case 1.} \label{figure:LimSupsCase1}
\end{figure}

So assume that the $x_i$ escape to infinity.

{\bf Case 2:} Suppose for the moment that the $A_i$ escape to infinity, i.e. that they are eventually disjoint from every bounded set. If $p \notin Q$ then we can find disjoint open sets $U$ and $V$ in $\overline{\bH^3}$ that contain $p$ and $Q$ respectively. Since sets in $P$ close to infinity correspond to small sets in $S^2_\infty$ (Lemma~\ref{lemma:FlowlinesEventuallySmall}) it follows that $\pi^{-1}(x_i)$ and $\pi^{-1}(A_i)$ are eventually contained in $U$ and $V$ respectively (see Figure~\ref{figure:LimSupsCase2}). But this contradicts the fact that $\pi^{-1}(A_i)$ separates $\pi^{-1}(x_i)$ from $\pi^{-1}(x_{i-1})$ for each $i$. Thus $p \in Q$.

\begin{figure}[ht]
	\labellist
	\small\hair 2pt
	\pinlabel $x_1$ [b] at 20 35
	\pinlabel $x_2$ [b] at 54 49
	\pinlabel $x_3$ [b] at 78 60
	\pinlabel $\cdots$ [c] at 98 119
	\pinlabel $\cdots$ [c] at 142 140
	
	\pinlabel $\pi$ [b] at 167 127
	
	\pinlabel $p$ [b] at 216 103
	
	\pinlabel $Q$ [b] at 261 120
	
	\pinlabel $U$ [tl] at 227 89
	\pinlabel $V$ [tl] at 291 96

	\endlabellist
	\centering
		\includegraphics{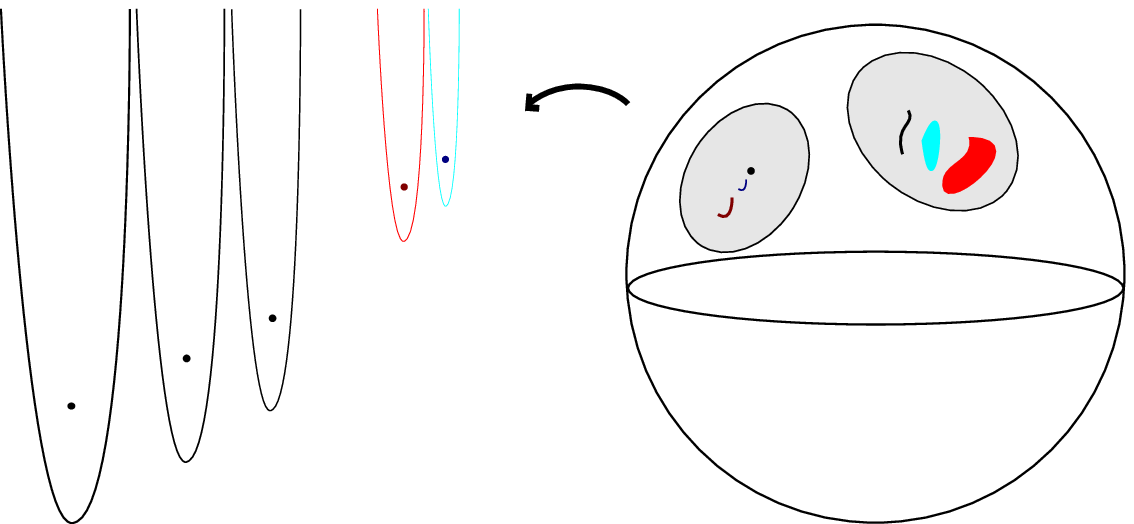}
	\caption{Case 2.} \label{figure:LimSupsCase2}
\end{figure}

{\bf Case 3:} Now suppose that the $A_i$ do not escape to infinity. We will cut a compact set out of each $U_i$ to make the $A_i$ escape to infinity. In general this might change $Q = \limsup A_i$ but by choosing carefully we can ensure that $Q$ shrinks.

Choose an exhaustion $(D_j)_{j=1}^\infty$ of $P$ by compact discs, each of which intersects \emph{all} of the $U_i$. Note that for each $j$, the sets $D_j \cap U_i$ are all contained in the bounded set $D_j$. The argument in case 1 shows that that for a bounded sequence of points $y_i \in U_i$, $e^+(y_i)$ is eventually close to $Q$. So for fixed $j$, $e^+(D_j \cap U_i)$ is eventually close to $Q$ by the same argument. 

It follows that we can find a ``diagonal sequence'': For each $j$ choose an integer $i(j)$ sufficiently large so that (a) $x_{i(j)}$ is outside of $D_j$ and such that (b) $\limsup_{j \to \infty} e^+(D_j \cap U_{i(j)}) \subset Q$. Now for each $j$, let $y_j := x_{i(j)}$, let $V_j$ be the component of $P \setminus (D_j \cup U_{i(j)})$ that contains $y_j$, and let $B_j := \Fr V_j$. Also, set $Q' := \limsup_{j \to \infty} B_j$ and observe that since each point in $B_j$ is contained in either $A_{i(j)}$ or $(D_j \cap U_{i(j)})$, property (b) ensures that $Q' \subset Q$. See Figure~\ref{figure:LimSupsCase3}.

Now the $B_j$ escape to infinity, so by case 2 $\lim e^+(y_j) \subset Q'$ and therefore $p \in Q$ as desired.

\begin{figure}[ht]
	\labellist
	\small\hair 2pt
	
	\pinlabel $x_{i(j)}$ [b] at 120 95
	\pinlabel $\color{blue}{B_{j}}$ [l] at 125 83

	\pinlabel {$D_j \cap U_{i(j)}$} [l] at 170 58
	
	\pinlabel {$\Fr D_j$} [l] at 168 46
	
	
	\pinlabel $\cdots$ [c] at 99 128
	\pinlabel $\cdots$ [c] at 137 128
	\pinlabel $\vdots$ [c] at 100 49
	\pinlabel $\vdots$ [c] at 100 85

	\endlabellist
	\centering
		\includegraphics{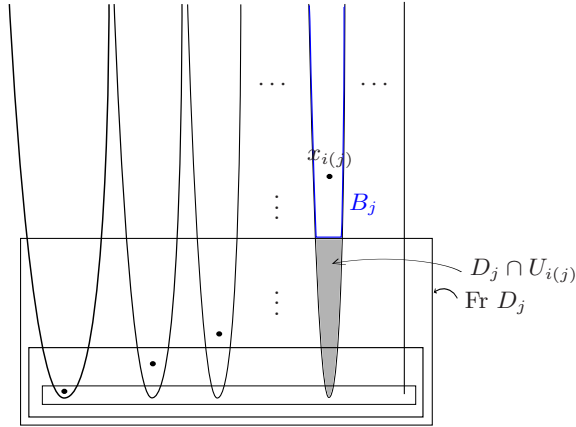}
	\caption{Case 3: Fixing the $A_i$.} \label{figure:LimSupsCase3}
\end{figure}
\end{proof}

\begin{corollary} \label{corollary:FinitelyManyCompCpnts}
Let $B, C \subset S^2_\infty$ be disjoint compact sets. Then $(e^+)^{-1}(B)$ intersects at most finitely many components of $P \setminus (e^+)^{-1}(C)$.
\end{corollary}
\begin{proof}
Otherwise we could find a sequence of points $b_i$ in $(e^+)^{-1}(B)$, each of which is contained in a different component of $P \setminus (e^+)^{-1}(C)$. Then $\limsup e^+(b_i) \in C$ by the preceeding lemma. But $\limsup e^+(b_i) \in B$, so $B \cap C \neq \emptyset$, a contradiction.
\end{proof}

\subsection{Closures of decomposition elements}

Recall (\cite{Frankel}, Section~3) that the \emph{positive decomposition} of $P$ is defined as
\[ \cD^+ := \{ \text{connected components of } (e^+)^{-1}(p) | p \in S^2_\infty \}. \]
Each decomposition element $K \in \cD^+$ is an \emph{unbounded continuum} in $P$, i.e. a noncompact, closed, and connected set. The action of $\pi_1(M)$ on $P$ permutes the elements of $\cD^+$. The \emph{negative decomposition} $\cD^-$ is defined similarly using the $e^-$ map, and we set $\cD := \cD^+ \cup \cD^-$.

An unbounded continuum $K$ in the plane has a naturally defined set of \emph{ends} $\cE(K)$, where an end $\kappa \in \cE(K)$ is a map such that for each bounded set $D \in P$, $\kappa(D)$ is an unbounded component of $K \setminus D$, and if $D \subset D'$ then $\kappa(D') \subset \kappa(D)$.

If $K \in \cD^+$ is a positive decomposition element and $L \in \cD^-$ is a negative decomposition element then $K \cap L$ is compact.\footnote{\cite{Frankel}, Lemma~3.3}

The set of \emph{ends}
\[ \cE := \bigcup_{K \in \cD} \cE(K) \]
comes with a natural circular order. For a triple $\kappa, \lambda, \mu \in \cD$ choose a bounded open disc $D \subset P$ such that $\kappa(D)$, $\lambda(D)$, and $\mu(D)$ are disjoint and mutually nonseparating (i.e. no one separates the other two) and let $\gamma$ be an oriented arc from $\kappa(D)$ to $\mu(D)$. Then $\kappa, \lambda, \mu$ is positively ordered if $\lambda(D)$ is on the positive side of $\gamma$ and negatively ordered otherwise. If $\kappa, \mu \in \cE$ then the \emph{open interval} $(\kappa, \mu)$ is the set of $\lambda \in \cE$ such that $\kappa, \lambda, \mu$ is positively ordered.

The action of $\pi_1(M)$ on $P$ permutes the elements of $\cE$ preserving the circular order.

There is a natural \emph{order completion} $\widetilde{\cE}$ that is homeomorphic to a closed subset of a circle and $\cE$ is dense in $\widetilde{\cE}$. We can collapse complementary intervals in $\widetilde{\cE}$ form the universal circle $S^1_u$, and $\widetilde{\cE}$ surjects onto this.\footnote{See \cite{Frankel}, Construction~7.3.} We'll denote by $\phi: \cE \to S^1_u$ the composition of this inclusion and surjection. Both $\phi(\cE^+)$ and $\phi(\cE^-)$ are dense in $S^1_u$, where $\cE^+$ and $\cE^-$ are the ends that come from positive and negative decomposition elements respectively, and there are at most countably many ends in the preimage of each point. The action of $\pi_1(M)$ on $\cE$ extends to $\widetilde{\cE}$ and induces an action on $S^1_u$ by homeomorphisms.

The spaces $P$ and $S^1_u$ glue together to form the \emph{end compactification} $\overline{P}$, which is homeomorphic to a closed disc with interior $P$ and boundary $S^1_u$. The actions of $\pi_1(M)$ on $P$ and $S^1_u$ match up to form an action on $\overline{P}$.\footnote{\cite{Frankel}, Theorem~7.9}

See \cite{Frankel} for a more complete discussion of ends, circular orders, and universal circles. In order to prove the extension theorem we will need to work out the details of how decomposition elements look with respect to the end compactification. In particular, distinct positive/negative decomposition elements are disjoint in $P$ but their closures in $\overline{P}$ might intersect. We'll concentrate on $\cD^+$ for notational simplicity, though everything we say applies to $\cD^-$. 

First, we'll recall how the separation properties of decomposition elements in the orbit space reflect the separation properties of their ends. Two subsets $A$ and $B$ of a circularly ordered set $S$ are said to \emph{link} if there are $a, a' \in A$ and $b, b' \in B$ such that $b \in (a, a')$ and $b' \in (a', a)$. The set $C \subset S$ is said to \emph{separate} $A$ and $B$ if there are $c, c' \in C$ such that $A \subset (c, c')$ and $B \subset (c', c)$.

\begin{lemma} [\cite{Frankel}, Proposition 6.9] \label{lemma:Nonlinking}
Let $K, L \in \cD^+$. Then $\cE(K)$ and $\cE(L)$ do not link.

Let $K, L, M \in \cD^+$. Then $K$ separates $L$ from $M$ if and only if $\cE(K)$ separates $\cE(L)$ from $\cE(M)$.
\end{lemma}

The following two lemmas are the basic results we'll need about the closures of decomposition elements.

\begin{lemma}[\cite{Frankel}, Lemma 7.8] \label{lemma:EndsDense}
Let $K \in \cD^+$. Then
\[ \overline{K} = K \cup \overline{\phi(\cE(K)}) \]
where the closure is taken in $\overline{P}$.
\end{lemma}

\begin{lemma} \label{lemma:CountablyManyDecompElts}
For each point $x \in \partial \Orb$ there are at most countably many decomposition elements $K \in \cD^+$ with $\overline{K} \ni x$.
\end{lemma}
\begin{proof}
By Lemma~\ref{lemma:EndsDense}, each decomposition element $K$ whose closure intersects $x$ either has an end at $x$ (i.e. there is a $\kappa \in \cE(K)$ with $\phi(\kappa) = x$) or has ends approaching $x$ (i.e. there are $\kappa_i \in \cE(K)$ with $\phi(\kappa_i) \to x$ as $i \to \infty$). Recall that there are countably many decomposition elements with an end at $x$, so it suffices to show that there are countably many with ends approaching $x$. In fact there are at most two.

We will say that a decomposition element $K \in \cD^+$ has ends approaching $x$ in the \emph{positive direction} if there are $\kappa_i \in \cE(K)$ such that $\phi(\kappa_i) \to x$ as $i \to \infty$ and $\phi(\kappa_{i+1}) \subset (\phi(\kappa_i), x)$ for all $i$. There can be at most one decomposition element with ends approaching $x$ in the positive direction since if there were two then their ends would link, contradicting Lemma~\ref{lemma:Nonlinking}. Similarly, there is at most one decomposition element with ends approaching $x$ in the negative direction.
\end{proof}

\subsection{Unions of decomposition elements}

\begin{lemma} \label{lemma:SomeComponentSeparates}
Let $A$ be a closed subset of the plane that separates the points $x$ and $y$. Then some component of $A$ separates $x$ and $y$.
\end{lemma}
\begin{proof}
Let $U$ be the component of $P \setminus A$ that contains $x$, and let $V$ be the component of $P \setminus \overline{U}$ that contains $y$. Then $\Fr V$ is connected because the plane satisfies the Brouwer property (see \cite{Wilder}, Definition I.4.1): if $M$ is a closed, connected subset of the plane $P$ and $V$ is a component of $P \setminus M$ then $\Fr V$ is connected. But $\Fr V \subset A$ so take the component of $A$ containing $V$.
\end{proof}

Recall that if $(A_i)_{i=1}^\infty$ is a sequence of sets in a topological space then $\liminf A_i$ is the set of points $x$ such that every neighborhood of $x$ intersects all but finitely many of the $A_i$. If $\liminf A_i$ is nontrivial and agrees with $\limsup A_i$ then the $A_i$ are said to converge in the Hausdorff sense and we write $\lim A_i = \limsup A_i = \liminf A_i$.\footnote{See \cite{HockingYoung}} If the $A_i$ live in the plane and infinitely many intersect some compact set then some subsequence converges in the Hausdorff sense.

\begin{lemma} \label{lemma:FrontierOfComplement}
Let $A \subset P$ be a closed, connected set that is a union of positive decomposition elements and let $U$ be a component of $P \setminus A$. Then $\Fr U$ is contained in a single decomposition element.
\end{lemma}
\begin{proof}
Note that $\Fr U$ is connected by the Brouwer property, so it suffices to show that $e^+(\Fr U)$ is a single point.

Suppose that $e^+(\Fr U)$ is not a single point. Then it is infinite so we can choose three points $k, l, m \in \Fr U$ whose images under $e^+$ are distinct and let $K, L, M \subset A$ be the positive decomposition elements that contain them. Note that these are mutually nonseparating, so by Lemma~\ref{lemma:Nonlinking} $\cE(K)$, $\cE(L)$, and $\cE(M)$ are mutually nonseparating in the circular order on $\cE$.

Assume that $K, L, M$ are ordered so that $\cE(L)$ is contained in the positively oriented interval between $\cE(K)$ and $\cE(M)$ and each decomposition element contained in $U$ has ends contained in the positively oriented interval between $\cE(M)$ and $\cE(K)$.

Now let $n_i \in P$ be a sequence of points in $U$ that approach some point in $L$ and let $N_i$ be the decomposition element containing $n_i$ for each $i$. Note that infinitely many of the $N_i$ intersect a compact set (for example, any arc from $K$ to $M$) so we can take a Hausdorff convergent subsequence. See Figure~\ref{figure:FrontierOfComplement}. Let $N_\infty$ be the limit of this subsequence. Then $N_\infty \cup L$ separates $K$ from $M$. Indeed, if $\gamma$ is any arc from $K$ to $M$ that avoids $L$ then the $n_i$ are eventually on the positive side of $\gamma$ and the $N_i$ have ends between $\cE(M)$ and $\cE(K)$, so the $N_i$ eventually intersect $\gamma$. So by Lemma~\ref{lemma:SomeComponentSeparates} some component $B$ of $N_\infty \cup L$ separates $K$ from $M$. But then $\cE(B)$ separates $\cE(K)$ from $\cE(M)$, a contradiction. So $e^+(\Fr U)$ is a single point as desired.

\begin{figure}[ht]
	\labellist
	\small\hair 2pt
	
	\pinlabel $\color[gray]{0.7}{A}$ [c] at 51 135
	\pinlabel $\color{blue}{\Fr U}$ [tl] at 24 61
	\pinlabel $K$ [br] at 146 98
	\pinlabel $L$ [l] at 95 149
	\pinlabel $M$ [bl] at 32 97
	
	\pinlabel $\color{red}{N_\infty}$ [r] at 89 50

	\endlabellist
	\centering
		\includegraphics{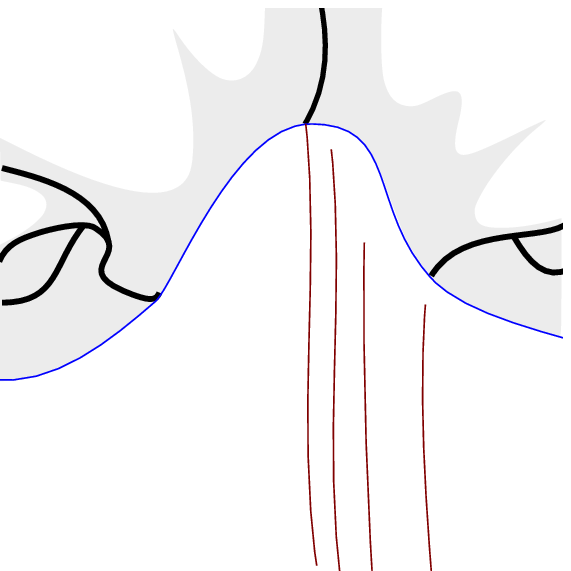}
	\caption{} \label{figure:FrontierOfComplement}
\end{figure}
\end{proof}

Combining the last two lemmas:

\begin{corollary} \label{corollary:DecompEltSeparates}
Let $A \subset P$ be a closed union of positive decomposition elements. If $A$ separates the points $x, y \in P$ then some decomposition element $K \subset A$ separates $x$ and $y$.
\end{corollary}

\subsection{Extending the endpoint maps}
We can now prove our main theorem.

\begin{theorem}
The map $e^+: P \to S^2_\infty$ has a unique continuous extension to $\overline{P}$.
\end{theorem}
\begin{proof}
For each $x \in \partial \overline{P}$ let $(x_i)_{i=1}^\infty$ be a sequence of points in $P$ that converges to $x$. After taking subsequence we can assume that $e^+(x_i) \to p$ for some $p \in S^2_\infty$ as $i \to \infty$. Set $e^+(x) = p$.

To see that this is well-defined, suppose we have two such sequences $(x_i)_{i=1}^\infty$ and $(y_i)_{i=1}^\infty$ converging to $x$ with $e^+(x_i) \to p$ and $e^+(x_i) \to q$ as $i \to \infty$ for $p \neq q$. Let $A'$ be a simple closed curve separating $p$ and $q$. Then $A = (e^+)^{-1}(A')$ separates the $x_i$ from the $y_i$ for large enough $i$. By Lemma~\ref{lemma:LimSups} the $x_i$ are eventually contained in a single component of $P \setminus A$ as are the $y_i$. So by Corollary~\ref{corollary:DecompEltSeparates} some decomposition element $K \in \cD^+$ separates $\{x_i\}_{i = I}^\infty$ from $\{y_i\}_{i = I}^\infty$ for some $I$. Therefore, $\overline{K}$ separates $\{x_i\}_{i = I}^\infty$ from $\{y_i\}_{i = I}^\infty$ in $\overline{P}$. This implies that $x \in \overline{K}$.

There are uncountably many disjoint simple closed curves separating $p$ and $q$, hence there are uncountably many distinct decomposition elements $K$ with $x \in \overline{K}$. This contradicts Lemma~\ref{lemma:CountablyManyDecompElts}, so the extension of $e^+$ to $\overline{P}$ is well-defined. Continuity and uniqueness are immediate.
\end{proof}

We can extend the negative endpoint map similarly.

\begin{proposition}
The maps $e^\pm: \overline{P} \to S^2_\infty$ agree on the boundary $S^1_u$.
\end{proposition}
\begin{proof}
Let $x \in \partial \overline{P}$ and choose a sequence of points $x_i \in P$ converging to $x$. By Lemma~\ref{lemma:FlowlinesEventuallySmall}, the distance between $e^+(x_i)$ and $e^-(x_i)$ approaches $0$ as $i \to \infty$.
\end{proof}

This completes the Extension Theorem. Our generalization of the Cannon-Thurston theorem follows easily:

\begin{theorem}
Let $\psi: P \to \bH^3$ be a section of the quotient map $\bH^3 \to P$. Then there is a unique continuous extension of $\psi$ to $\overline{P}$.
\end{theorem}
\begin{proof}
Just set $\psi$ on the boundary $S^1_u$ to be equal to $e^+$. That this is continuous follows from Lemma~\ref{lemma:FlowlinesEventuallySmall}.
\end{proof}

\section{Closed orbits and M\"obius-like groups}
If $M$ is a closed hyperbolic 3-manifold then each $g \in \pi_1(M)$ acts on $\bH^3 \cup S^2_\infty$ as a translation and rotation along a geodesic axis. The endpoints of this axis in $S^2_\infty$ are the fixed points for the action of $g$. For each such $g$ we will label the attracting fixed point $a_g$ and the repelling fixed point $r_g$. 

\begin{lemma} \label{lemma:FlowlinesAtFixedPoints}
Let $\rF$ be a quasigeodesic flow on a closed hyperbolic 3-manifold $M$ and let $g \in \pi_1(M)$. Then $\rF$ has a closed orbit in the free homotopy class represented by $g$ if and only if there has a flowline with an endpoint at either $a_g$ or $r_g$.
\end{lemma}
\begin{proof}
Suppose that there is a closed orbit in the free homotopy class of $g$. Then some lift $\gamma$ of this orbit to the universal cover is fixed by $g$, so the endpoints of $\gamma$ are at $a_g$ and $r_g$.

Conversely, suppose that a flowline $\gamma$ in the lifted flow has an endpoint at $r_g$ (replace $g$ by $g^{-1}$ if it has a fixed point at $a_g$). Then the endpoints of $g^n(\gamma)$ approach $a_g$ and $r_g$ as $n \to \infty$, so the $g^n(\gamma)$ all intersect a compact set in $\bH^3$. In other words, the point $\pi(\gamma) \subset P$ in the orbit space corresponding to $\gamma$ has a bounded forward orbit under $g$. The Brouwer plane translation theorem\footnote{\cite{Franks}} implies that $g$ must fix a point in $P$, and this fixed point corresponds to a closed orbit in the free homotopy class of $g$.
\end{proof}

The Closed Orbits Theorem is an immediate consequence of the following.

\begin{theorem}
Let $M$ be a closed hyperbolic 3-manifold with a quasigeodesic flow $\rF$. Suppose that the action of $g \in \pi_1(M)$ on $S^1_u$ is not conjugate to a hyperbolic M\"obius transformation. Then $\rF$ has a closed orbit in the free homotopy class represented by $g$.
\end{theorem}
\begin{proof}
Suppose $f$ is a homeomorphism of the circle and $x \in S^1$ is an isolated fixed point for $f$. Then $x$ is either an attracting, repelling, or neither, in which case we call it indifferent. Each homeomorphism of the circle can be classified up to conjugacy by whether it has
\begin{enumerate}
	\item no fixed points,
	\item one indifferent fixed point (i.e. it is conjugate to a parabolic M\"obius transformation),
	\item two indifferent fixed points,
	\item two fixed points in an attracting-repelling pair (i.e. it is conjugate to a parabolic M\"obius transformation), or
	\item more than two fixed points.
\end{enumerate}

In \cite{Frankel}, we showed that if the action of $g$ fits into categories 1, 3, or 5 then the flow $\rF$ has a closed orbit in the free homotopy class of $g$. So we just need to extend this to category 2.

Suppose that $g$ acts on $S^1_u$ with a single fixed point $x$. If there are no closed orbits in the homotopy class represented by $g$ then by Lemma~\ref{lemma:FlowlinesAtFixedPoints}, there are no flowlines with an endpoint at either $a_g$ or $r_g$, so $A = (e^+)^{-1}(a_g)$ and $R = (e^-)^{-1}(r_g)$ are both contained in the boundary $S^1_u$ of $\overline{P}$. But $A$ and $B$ are invariant sets, so they must both contain the fixed point $x$, a contradiction.
\end{proof}

\end{document}